\documentclass[12pt,reqno]{amsart}
\usepackage[margin=2in]{geometry}
\usepackage[utf8]{inputenc}
\usepackage{graphicx}
\usepackage{bookmark}
\usepackage{orcidlink}

\usepackage{comment}
\graphicspath{ {./images/} }
\usepackage{amsmath, amsthm, amssymb, mathrsfs, mathtools, mathdots}
\usepackage{tikz-cd}
\usepackage[framemethod=TikZ]{mdframed}

\usepackage{thmtools}

\numberwithin{equation}{section}

\usepackage{subfiles}
\usepackage{xcolor, colortbl}
\usepackage{soul}
\usepackage{graphicx}
\usepackage{marginnote}
\usepackage{fancyhdr}
\usepackage{hyperref}
\usepackage{mathpazo}
\hypersetup{
	colorlinks=true,
	linkcolor=blue,
	filecolor=magenta,      
	urlcolor=cyan,
	pdftitle={Classification of Finite Groups With Equal Left and Right Quotient Sets},
	pdfpagemode=FullScreen,
	citecolor=blue,
}

\makeatletter
\def\thm@space@setup{\thm@preskip=1em
	\thm@postskip=1em}
\makeatother

\theoremstyle{plain}
\newtheorem{theorem}{Theorem}[section]
\newtheorem{lemma}[theorem]{Lemma}
\newtheorem{corollary}[theorem]{Corollary}

\newtheorem{proposition}[theorem]{Proposition}

\theoremstyle{definition}
\newtheorem{definition}[theorem]{Definition}
\theoremstyle{remark}

\theoremstyle{plain}

\newcommand{\Z}{\mathbb{Z}}

\newcommand{\Q}{\mathbb{Q}}
\newcommand{\Aut}{\mathrm{Aut}}
\newcommand{\Gal}{\mathrm{Gal}}
\newcommand{\lcm}{\mathrm{lcm}}
\newcommand{\sep}{\mathrm{sep}}
\newcommand{\Stab}{\mathrm{Stab}}
\newcommand{\scr}{\mathscr}
\newcommand{\cal}{\mathcal}
\newcommand{\nbhd}{\mathrm{nbhd}}

\title{The Index Problem for Subgroup Intersections}
\author{Haran Mouli\,\orcidlink{0009-0000-5891-6618}}\email{hmouli@wisc.edu}
\address{Department of Mathematics, University of Wisconsin, Madison, WI, 53706}

\date{\today}

\begin{document}
	
	\begin{abstract}
		In \cite{DDL} and \cite{DM}, Drungilas et al. study the problem of which triples of positive integers $(a, b, c)$ can be realized as $([E: \Q], [F: \Q], [EF: \Q])$, where $E$ and $F$ are number fields, using techniques from field theory. We shall study this problem rephrased in the language of groups using the Galois correspondence to simplify and generalize their results.
	\end{abstract}
	
	\maketitle 
	
	\tableofcontents

	\section{Introduction}
	
	\begin{definition}\label{def: Index-Realizable}
		We say that a triple $(a, b, c)$ of positive integers is \textbf{index-realizable} if there exists a group $G$ with finite index subgroups $H$ and $K$ such that $a = [G: H]$, $b = [G: K]$, and $c = [G: H \cap K]$.
	\end{definition}
	
	Note that in \autoref{def: Index-Realizable}, we may assume without loss of generality that $G$ is a \textit{finite} group by replacing all of $G, H, K, H \cap K$ with their quotients by a finite index normal subgroup of $G$ contained in $H \cap K$.
	
	In \cite{DDL} and \cite{DM}, the authors study the notion of \textbf{compositum-feasible} triples over $\Q$: these are triples of positive integers $(a, b, c)$ for which there exist number fields $E$ and $F$ such that $a = [E: \Q]$, $b = [F: \Q]$, and $c = [EF: \Q]$. By the Galois correspondence, this is equivalent to finding subgroups $H$ and $K$ as in \autoref{def: Index-Realizable} with the absolute Galois group $G = \Gal(\bar{\Q}/\Q)$. In particular, all compositum-feasible triples over $\Q$ (or any other pefect field) are index-realizable. The converse is true if one assumes a positive answer to the inverse Galois problem.
	
	Consequently, one can prove slightly stronger results by focusing on the group-theoretic problem of classifying index-realizable triples. Indeed, if $(a, b, c)$ is not index-realizable, then it is certainly not compositum-feasible over $\Q$. On the other hand, whenever we prove that $(a, b, c)$ is index-realizable, we shall do so by example, and the interested reader can verify that the groups $G$ used in our examples are all Galois groups over $\Q$. One might find class field theory or Shafarevich's theorem on solvable Galois groups useful for this purpose.
	
	In Section \ref{sec: Basic Results}, we establish some simple facts about index-realizable triples. In Section \ref{sec: Coset Intersection Graph}, we introduce our main tool, the \textbf{coset intersection graph}, which will encode the problem of index-realizability in terms of graph theory. We then provide simpler proofs to (generalizations of) results in \cite{DDL} and \cite{DM}. Finally, in Section \ref{sec: Small Triples}, we shall classify all index-realizable triples $(a, b, c)$ with $\min(a, b) \leqslant 15$ using the tools from the previous sections.

    \subsection*{Acknowledgements} The author thanks Brian Lawrence for his helpful feedback, and Nigel Boston for directing him to \cite{DM}.

	\section{Basic Results}\label{sec: Basic Results}
	
	\begin{proposition}\label{prop: Necessary Conditions}
		If $(a, b, c)$ is index-realizable, then $\lcm(a, b) \mid c$ and $c \leqslant ab$.
	\end{proposition}
	
	\begin{proof}
		Let $G$ be a finite group with subgroups $H$ and $K$ realizing the triple $(a, b, c)$. The first part follows from Lagrange's theorem since $|H \cap K|$ must divide both $|H|$ and $|K|$. For the second part, apply the formula:
		\[|HK| = \frac{|H| \cdot |K|}{|H \cap K|}\]
		and use the fact that $|HK| \leqslant |G|$.
	\end{proof}
	
	The converse to \autoref{prop: Necessary Conditions} is false. It is proven in \cite{DDL} that the triple $(5, 5, 15)$ is the smallest triple satisfying the constraints in \autoref{prop: Necessary Conditions} which is not compositum-realizable over $\Q$. We shall prove in Section \ref{sec: Coset Intersection Graph} that it is an not index-realizable triple either.
	
	\begin{proposition}\label{prop: Products of Triples}
		If $(a, b, c)$ and $(a', b', c')$ are both index-realizable, then so is $(aa', bb', cc')$.
	\end{proposition}
	
	\begin{proof}
		If the triples $(a, b, c)$ and $(a', b', c')$ are realized by $(G, H, K)$ and $(G', H', K')$, then the product $(aa', bb', cc')$ is realized by $(G \times G', H \times H', K \times K')$.
	\end{proof}
	
	\begin{lemma}\label{lem: (a, b, ab)}
		For any positive integers $a$ and $b$, the triple $(a, b, ab)$ is index-realizable.
	\end{lemma}
	
	\begin{proof}
		Use $G = \Z/a \times \Z/b$ with subgroups $H = \Z/b$ and $K = \Z/a$.
	\end{proof}
	
	The following proposition generalizes \autoref{lem: (a, b, ab)}:
	
	\begin{proposition}\label{prop: (a, b, c) with lcm(a, b) dividing c dividing ab}
		If $(a, b, c)$ is a triple of positive integers with $\lcm(a, b) \mid c$ and $c \mid ab$, then $(a, b, c)$ is index-realizable.
	\end{proposition}
	
	\begin{proof}
		Let $d = \gcd(a, b)$, $a = dm$, and $b = dn$. We can rewrite the constraints in the proposition as $dmn \mid c$ and $c \mid d^2mn$. Thus, there exists a positive integer $e$ such that $e \mid d$ and $c = demn$. Let $f$ be such that $d = ef$; the triple $(a, b, c)$ is the product of the triples $(e, e, e^2)$, $(f, f, f)$, and $(m, n, mn)$. The first and third triples are index-realizable by \autoref{lem: (a, b, ab)}, and so is the second since we can choose $G = \Z/f$ with $H$ and $K$ both trivial. The proposition follows from \autoref{prop: Products of Triples}.
	\end{proof}
	
	\begin{theorem}\label{thm: Inseparable Extensions}
		Let $k$ be any field, not necessary perfect. If $E/k$ and $F/k$ are finite field extensions, then the triple $([E: k], [F: k], [EF: k])$ is index-realizable.
	\end{theorem}
	
	\begin{proof}
		For any finite extension $L/k$, let $L^{\sep}/k$ be the maximal separable subextension. Then, we can decompose $([E: k], [F: k], [EF: k])$ as the product:
		\[([E^{\sep}: k], [F^{\sep}: k], [(EF)^{\sep}: k]) \cdot (p^{\alpha}, p^{\beta}, p^{\gamma})\]
		where $p = \max(1, \mathrm{char}(k))$ and $p^{\alpha}, p^{\beta}, p^{\gamma}$ are the inseparable degrees of $E/k$, $F/k$, and $EF/k$ respectively. The first triple is index-realizable by the Galois correspondence and the second is index-realizable by \autoref{prop: (a, b, c) with lcm(a, b) dividing c dividing ab}. It follows that the triple $([E: k], [F: k], [EF: k])$ is index-realizable by \autoref{prop: Products of Triples}.
	\end{proof}
	
	\begin{corollary}
		For a triple of positive integers $(a, b, c)$, the following are equivalent:
		\begin{enumerate}
			\item $(a, b, c)$ is index-realizable.
			\item $(a, b, c)$ is compositum-feasible over some number field $k$, depending on $(a, b, c)$.
			\item $(a, b, c)$ is compositum-feasible over some field $k$, depending on $(a, b, c)$.
		\end{enumerate}
	\end{corollary}
	
	\begin{proof}
		We have $(1) \Rightarrow (2)$ since every finite group can be realized as a Galois group over some number field. $(2) \Rightarrow (3)$ is trivial and $(3) \Rightarrow (1)$ is \autoref{thm: Inseparable Extensions}.
	\end{proof}
	
	Finally, we include some specific examples of index-realizable triples.
	
	\begin{lemma}\label{lem: (n, n, n(n-1))}
		For any positive integer $n > 1$, the triple $(n, n, n(n-1))$ is index-realizable.
	\end{lemma}
	
	\begin{proof}
		Let $G = S_n$, $H = \Stab(1)$, and $K = \Stab(2)$.
	\end{proof}
	
	\begin{lemma}\label{lem: (n, n, np) if p divides phi(n)}
		For any positive integer $n$ and prime divisor $p \mid \varphi(n)$, the triple $(n, n, np)$ is index-realizable.
	\end{lemma}
	
	\begin{proof}
		Let $G = \Z/n \rtimes \Z/p$ be a non-trivial semidirect product; such semidirect products exist since $p \mid \varphi(n)$. Pick $H$ and $K$ to be any two distinct complements to the normal subgroup $\Z/n$.
	\end{proof}

	\section{Coset Intersection Graph}\label{sec: Coset Intersection Graph}
	
	\begin{definition}
		Let $H$ and $K$ be finite index subgroups of a group $G$. The \textbf{coset intersection graph} for $(G, H, K)$ is the bipartite graph $\cal{G}$ with:
		\begin{itemize}
			\item partitions as the coset spaces $\cal{H} = G/H$ and $\cal{K} = G/K$.
			\item $(yH, zK) \in E(\cal{G})$ if and only if $yH \cap zK \neq \emptyset$.
		\end{itemize}
	\end{definition}
	
	\begin{proposition}
		Let $\cal{G}$ be the coset intersection graph corresponding to $(G, H, K)$. The map $G/(H \cap K) \to E(\cal{G})$ given by:
		\[x(H \cap K) \mapsto (xH, xK)\]
		is a well-defined and canonical bijection.
	\end{proposition}
	
	\begin{proof}
		It is easy to see that the map is well-defined, and hence, canonical. It is injective since $xH \cap xK = x(H \cap K)$. On the other hand, for any $(yH, zK) \in E(\cal{G})$, we can pick some $x \in yH \cap zK$ and relabel the edge as $(xH, xK)$, proving that the map is surjective, and hence a bijection.
	\end{proof}
	
	In particular, observe that any vertex $yH$ in $\cal{H}$ has incident edges as the $(H \cap K)$-cosets contained in $yH$, and similarly for $\cal{K}$. Thus, the vertices in $\cal{H}$ and $\cal{K}$ have degrees $[H: H \cap K]$ and $[K: H \cap K]$ respectively.
	
	\begin{definition}
		For a bipartite graph $\Gamma$ with partitions $\eta$ and $\kappa$, $\Aut(\Gamma)$ is the group of automorphisms of $\Gamma$ stabilizing $\eta$ and $\kappa$ set-wise.
	\end{definition}
	
	If $\cal{G}$ is the coset intersection graph of $(G, H, K)$, then it is easy to check that the action of $G$ on $G/H$, $G/K$, and $G/(H \cap K)$ translates into an action of $G$ on $\cal{G}$ which is transitive on the edges. This implies that $\Aut(\cal{G})$ also acts transitively on $E(\cal{G})$.
	
	\begin{theorem}\label{thm: Groups to Graphs}
		A triple $(a, b, c)$ of positive integers is index-realizable if and only if there exists a bipartite graph $\Gamma$ with partitions $\eta$ and $\kappa$ such that:
		\begin{enumerate}
			\item $\Gamma$ has no isolated vertices.
			\item $|\eta| = a$, $|\kappa| = b$, and $|E(\Gamma)| = c$.
			\item $\Aut(\Gamma)$ acts transitively on $E(\Gamma)$.
		\end{enumerate}
	\end{theorem}
	
	\begin{proof}
		$(\Rightarrow)$ Let $(G, H, K)$ realize the triple $(a, b, c)$. If $\cal{G}$ is the corresponding coset intersection graph, then the action of $G$ on $G/H$, $G/K$, and $G/(H \cap K)$ translates into an action of $G$ on $\cal{G}$, which is transitive on $E(\cal{G}) = G/(H \cap K)$. This implies that $\Aut(\cal{G})$ also acts transitively on $E(\cal{G})$.
		
		$(\Leftarrow)$ Pick any vertices $h \in \eta$ and $k \in \kappa$. Let $(G, H, K) = (\Aut(\Gamma), \Stab(h), \Stab(k))$. By $(1)$ and $(3)$, we see that $\Aut(\Gamma)$ acts transitively on $\eta$ and $\kappa$, so we have $[G: H] = a$ and $[G: K] = b$. On the other hand, $H \cap K = \Stab(hk)$, so by $(2)$ and $(3)$, we have $[G: H \cap K] = c$. It follows that $(a, b, c)$ is index-realizable.
	\end{proof}

	We shall use \autoref{thm: Groups to Graphs} to prove that certain triples are not index-realizable. We begin with the following theorem; $(1)$ is Theorem 2 in \cite{DM} (where it is proven using field theory) and $(2)$ is a generalization of $(1)$. We provide independent proofs for both for the sake of clarity and leave it to the reader to choose their cups of tea.
	
	\begin{theorem}\label{thm: (n, nl, n(n-2)l)}
		Let $n > 2$ be a positive integer.
		\begin{enumerate}
			\item The triple $(n, n, n(n-2))$ is not index-realizable if and only if $n$ is odd and $n \neq 3$.
			\item For any positive integer $\ell$, the triple $(n, n\ell, n(n-2)\ell)$ is not index-realizable if and only if $n$ is odd and $(n-1) \nmid 2\ell$.
		\end{enumerate}
	\end{theorem}
	
	\begin{proof}[Proof of $(1)$]
		If $n = 2m$ is an even positive integer, then the triple $(n, n, n(n-2))$ is the product of $(m, m, m(m-1))$ and $(2, 2, 4)$, which are realizable by \autoref{lem: (n, n, n(n-1))} and \autoref{lem: (a, b, ab)}. It follows that $(n, n, n(n-2))$ is index-realizable when $n$ is even. We also know that the triple $(3, 3, 3)$ is obviously index-realizable.
		
		Now, let $n > 3$ be odd. Assume for the sake of contradiction that $(n, n, n(n-2))$ is index-realizable. By \autoref{thm: Groups to Graphs}, there exists a bipartite graph $\Gamma$ with partitions $\eta$ and $\kappa$ with $|\eta| = |\kappa| = n$ such that $\Aut(\Gamma)$ acts transitively on $E(\Gamma)$ and all vertices being $(n-2)$-regular. 
		
		Let $\bar{\Gamma}$ be the complementary bipartite graph; we see that $\Aut(\bar{\Gamma})$ acts transitively on $(\eta \times \kappa) \setminus E(\bar{\Gamma})$ and all vertices have degree $2$. This forces $\bar{\Gamma}$ to be the disjoint union of even cycle graphs, and since $\Aut(\bar{\Gamma})$ acts transitively on $\eta$ and $\kappa$, all these cycles must have equal length. 
		
		Say that $\bar{\Gamma}$ is the disjoint union of $\tfrac{n}{m}$ cycles of length $2m$, where $m > 1$. We know that $m \neq 2$ since $n$ is odd and $\tfrac{n}{m}$ is an integer. If $m \neq n$, we can find $h \in \eta$ and $k_1, k_2 \in \kappa$ such that $h$ and $k_1$ lie in the same cycle, $k_2$ lies in a different cycle, and $(h, k_1) \not\in E(\bar{\Gamma})$. This would be a contradiction since no automorphism can send $(h, k_1)$ to $(h, k_2)$. On the other hand, if $m = n$, then $\Aut(\bar{\Gamma})$ is the dihedral group of order $2n$, so $n(n-2) \leqslant 2n$, implying that $n \leqslant 4$, a contradiction.
	\end{proof}
	
	\begin{proof}[Proof of $(2)$]
		If $n = 2m$ is an even positive integer, then the triple $(n, n\ell, n(n-2)\ell)$ is the product of $(m, m, m(m-1))$ and $(2, 2\ell, 2\ell)$, which are realizable by \autoref{lem: (n, n, n(n-1))} and \autoref{prop: (a, b, c) with lcm(a, b) dividing c dividing ab}, so $(n, n\ell, n(n-2)\ell)$ is also realizable by \autoref{prop: Products of Triples}.
		
		Now, let $n$ be odd and assume that $(n, n\ell, n(n-2)\ell)$ is index-realizable. By \autoref{thm: Groups to Graphs}, there exists a bipartite graph $\Gamma$ with partitions $\eta$ and $\kappa$ with $|\eta| = n$ and $|\kappa| = n\ell$ such that $\Aut(\Gamma)$ acts transitively on $E(\Gamma)$, all vertices in $\eta$ have degree $(n-2)\ell$, and all vertices in $\kappa$ have degree $(n-2)$. If $\bar{\Gamma}$ be the complementary bipartite graph, then $\Aut(\bar{\Gamma})$ acts transitively on $(\eta \times \kappa) \setminus E(\bar{\Gamma})$, the vertices of $\eta$ have degree $2\ell$, and the vertices of $\kappa$ have degree $2$.
		
		We can equivalently record the data of $\bar{\Gamma}$ using a $2\ell$-regular multigraph $\scr{G}$ with $V(\scr{G}) = \eta$ and $E(\scr{G}) = \kappa$, where each $k \in \kappa$ connects its neighbors in $\bar{\Gamma}$. The condition on $\Aut(\bar{\Gamma})$ acting transitively on $(\eta \times \kappa)/E(\bar{\Gamma})$ now translates to $\Aut(\scr{G})$ acting transitively on the non-incident pairs in $\eta \times \kappa = V(\scr{G}) \times E(\scr{G})$.
		
		Since $n > 2$, any edge in $\kappa$ has a vertex not incident on it, so it follows that $\Aut(\scr{G})$ acts transitively on $E(\scr{G})$. This means that the connected components of $\scr{G}$ must be isomorphic. First, assume that $\scr{G}$ has more than one connected component. For any edge $h_1h_2 \in \scr{G}$ and vertex $h$ in another component, there is no automorphism sending $h$ to the component of $h_1h_2$. This means that the component containing $h_1h_2$ has no vertices other than $h_1$ and $h_2$, so every component has two vertices. This contradicts the fact that $n$ is odd.
		
		Next, consider the case where $\scr{G}$ is connected. Let $k = h_1h_2$ be an arbitrary edge. Since $\Stab(k)$ acts transitively on $\eta \setminus \{h_1, h_2\}$, every $h \in \eta \setminus \{h_1, h_2\}$ is adjacent to the same number of vertices in $\{h_1, h_2\}$.
		\begin{itemize}
			\item This number cannot be $0$ since $\scr{G}$ is connected and $n > 3$. 
			\item If the number is $1$, then since $\deg(h_1) = \deg(h_2)$, exactly half of all vertices in $\eta \setminus \{h_1, h_2\}$ must be connected to $h_1$. This implies that $n-2$, and consequently $n$, is even, which is a contradiction.
			\item If the number is $2$, this would mean that $h_1$ and $h_2$ are connected to all other vertices. Applying this to every edge $k$, we see that any pair of vertices in $\scr{G}$ are adjancent. Furthermore, since $n > 2$ and $\Aut(\scr{G})$ acts transitively on non-incident pairs, we can conclude that $\scr{G}$ has an equal number of edges between any two vertices; this number must be $\tfrac{2\ell}{n-1}$. Thus, we must have $(n-1) \mid 2\ell$, and in this case, $\scr{G}$ indeed satisfies the required conditions. \qedhere
		\end{itemize} 
	\end{proof}

	In our next theorem, $(1)$ is Theorem 3 in \cite{DM} (where it is proven using field theory) and $(2)$ is a generalization of $(1)$. Once again, we shall first prove $(1)$ since it is simpler, and then prove $(2)$. For any vertex $v$ in a graph, let $\nbhd(v)$ denote the set of vertices adjacent to $v$.
	
	\begin{theorem}\label{thm: (n, nl, npl)}
		Let $n$ be a positive integer and $p$ be a prime such that $p + 1 < n < 2p$.
		\begin{enumerate}
			\item The triple $(n, n, np)$ is not index-realizable.
			\item For any positive integer $\ell$, $(n, n\ell, np\ell)$ is index-realizable if and only if ${n \choose p} \mid n\ell$.
		\end{enumerate}
	\end{theorem}
	
	\begin{proof}[Proof of $(1)$]
		Assume for the sake of contradiction that $(n, n, np)$ is index-realizable. Let $\Gamma$ be a bipartite graph with partitions $\eta$ and $\kappa$ satisfying the conditions of \autoref{thm: Groups to Graphs} for the triple $(n, n, np)$; note that $\Gamma$ is $p$-regular. Pick any $hk \in E(\Gamma)$; we know that $[\Stab(k): \Stab(hk)] = p$, so any Sylow $p$-subgroup of $\Stab(k)$ is not contained in $\Stab(hk)$. Consequently, there exists $\sigma \in \Stab(k)$ with order equal to a power of $p$ which does not fix $h$. Since $p < n < 2p$, we see that $\sigma$ acts on $\nbhd(k)$ as a $p$-cycle and fixes the points in $\eta \setminus \nbhd(k)$.
		
		For any $k' \in \kappa$, since $p > \tfrac{n}{2}$, we have $\nbhd(k) \cap \nbhd(k') \neq \emptyset$. Furthermore, if $k'$ is fixed by $\sigma$, we must have $\nbhd(k') = \nbhd(k)$ since $\nbhd(k')$ must be $\sigma$-stable and $\sigma$ acts transitively on $\nbhd(k)$. This cannot be the case for all $k' \in \kappa$, so $\sigma$ has one $p$-cycle and $n-p$ fixed points in $\kappa$. Repeating the same argument:
		\begin{itemize}
			\item Every $\sigma$-fixed point in each partition is adjacent to the points in the $p$-cycle of $\sigma$ in the opposite partition.
			\item Each point not fixed by $\sigma$ in one partition must have $p - (n-p) = 2p-n$ neighbors not fixed by $\sigma$ in the opposite partition.
		\end{itemize}
		
		Since $n-p > 1$, we see that $\eta$ contains points with the same neighbor set. Since $\Aut(\Gamma)$ acts transitively on $\eta$, if we pick $h' \in \eta$ not fixed by $\sigma$, there must exist another $h'' \in \eta$ with $\nbhd(h'') = \nbhd(h')$. Note that $h''$ must also not be fixed by $\sigma$, so $h'' = \sigma^ih'$ for some $0 \leqslant i < p$. However, the set of $2p-n$ neighbors of $h'$ not fixed by $\sigma$ cannot be stable under $\sigma^i$ for any $0 \leqslant i < p$ since $p$ is prime. Thus, $\nbhd(\sigma^ih') \neq \nbhd(h')$ for any $0 \leqslant i < p$, yielding the required contradiction and proving that $(n, n, np)$ is not index-realizable.
	\end{proof}
	
	\begin{proof}[Proof of $(2)$]
		First, note that the case $n = p + 2$ is a special case of \autoref{thm: (n, nl, n(n-2)l)}, so assume that $p + 2 < n < 2p$. Let $\Gamma$ be a bipartite graph with partitions $\eta$ and $\kappa$ satisfying the conditions of \autoref{thm: Groups to Graphs} for the triple $(n, n\ell, np\ell)$. The vertices in $\kappa$ are $p$-regular, so for any $hk \in E(\Gamma)$, we have $[\Stab(k): \Stab(hk)] = p$. As before, we can find $\sigma \in \Stab(k)$ of order equal to a power of $p$ which acts on $\nbhd(k)$ as a $p$-cycle and fixes $\eta \setminus \nbhd(k)$.
		
		We claim that the action of $\Aut(\Gamma)$ on $\eta$ is primitive. Since $p$ is prime and the points in $\nbhd(k)$ form a $p$-cycle under $\sigma$, these points must all lie in the same block or in pairwise distinct blocks. The latter is impossible since it would imply that we have at least $p > \tfrac{n}{2}$ blocks, implying that there are $n$ blocks, which is absurd since $\Aut(\Gamma)$ acts non-trivially on $\eta$. In the former case, since the block containing $\nbhd(k)$ has more than $\tfrac{n}{2}$ points, it must contain all of $\eta$, proving that the action of $\Aut(\Gamma)$ on $\eta$ is primitive.
		
		Now, the primitive transitive action of $\Aut(\Gamma)$ on $\eta$ has a $p$-cycle, and $p < n - 2$, so by Jordan's theorem \cite{Jor}, the image of $\Aut(\Gamma)$ in $S_{\eta} = S_n$ must be either $S_n$ or $A_n$. In particular, $\Aut(\Gamma)$ must act $(n-2)$-transitively on $\eta$, and since $p < n - 2$, it also acts $p$-transitively. This means that every $p$-subset of $\eta$ is the image of $\nbhd(k)$ under some automorphism of $\Gamma$, and hence of the form $\nbhd(k')$ for some $k' \in \kappa$. Since $\Aut(\Gamma)$ acts transitively on $\kappa$, it follows that ${n \choose p} \mid n\ell$, and in this case, we have constructed a bipartite graph $\Gamma$ with the required properties.
	\end{proof}

	A common theme of the previous two theorems is that it is difficult for triples of the form $(n, n, nt)$ with large $t$ to be index-realized. The following proposition exhibits examples of such kind arising naturally from geometry.
	
	\begin{proposition}\label{prop: Geometric triples}
		Let $q$ be a prime power and $n > d$ be any positive integers. The triple:
		\[(a, b, c) = \left({n \choose d}_q, {n \choose d}_q, q^{d(n-d)}{n\choose d}_q\right)\]
		is index-realizable, where ${\bullet \choose \bullet}_q$ is the $q$-binomial coefficient.
	\end{proposition}
	
	\begin{proof}
		Let $\Gamma$ be the bipartite graph with $\eta$ and $\kappa$ corresponding to the subspaces of dimension $d$ and codimension $d$ in $\mathbb{F}_q^n$ respectively, such that for any $hk \in \eta \times \kappa$, we have $hk \in E(\Gamma)$ if and only if $h$ and $k$ are complements. For any subspace of dimension $d$, there are:
		\[\frac{(q^n-q^d)(q^n-q^{d+1}) \cdots (q^n-q^{n-1})}{(q^{n-d}-1)(q^{n-d}-q) \cdots (q^{n-d}-q^{n-d-1})} = q^{d(n-d)}\]
		complementary subspaces, so $\Gamma$ is $q^{d(n-d)}$-regular. Using \autoref{thm: Groups to Graphs}, we see that $(a, b, c)$ is index-realizable.
	\end{proof}

	\section{Small Triples}\label{sec: Small Triples}
	
	We conclude by classifying all triples $(a, b, c)$ of positive integers with $\min(a, b) \leqslant 10$ that are index-realizable. By \autoref{prop: Necessary Conditions} and \autoref{prop: (a, b, c) with lcm(a, b) dividing c dividing ab}, it suffices to look at triples $(a, b, c)$ such that $\lcm(a, b) \mid c $, $c < ab$, and $c \nmid ab$. 
	
	Our procedure will be to go through the cases $1 \leqslant a \leqslant 10$. Letting $b$ be arbitrary, we must look at $c = db$ for positive integers $1 < d < a$ such that $a \mid db$ and $c \nmid ab$, i.e. $d \nmid a$. We can filter the following cases:
	\begin{enumerate}
		\item If $d = a - 1$, then the condition $a \mid db$ implies that $a \mid b$, so $b = ae$ for some positive integer $e$. Then, the triple has the form $(a, ae, a(a-1)e)$, which is the product of $(a, a, a(a-1))$ and $(1, e, e)$. These triples are index-realizable, the former by \autoref{lem: (n, n, n(n-1))}. By \autoref{prop: Products of Triples}, $(a, b, c)$ is index-realizable.
		\item If $d = p$ is a prime and $p \mid \varphi(a)$, since $p \nmid a$ and $a \mid pb$, we obtain $a \mid b$, so $b = ae$ for some positive integer $e$. Then, the triple has the form $(a, ae, ape)$, which is the product of $(a, a, ap)$ and $(1, e, e)$.These triples are index-realizable, the former by \autoref{lem: (n, n, np) if p divides phi(n)}. By \autoref{prop: Products of Triples}, $(a, b, c)$ is index-realizable.
		\item If $d = a - 2$, then we must have $a \mid (a-2)b$, so $a \mid 2b$. If $a$ is odd, this implies $a \mid b$, so $b = ae$ for some positive integer $e$. In this case, we have $(a, ae, a(a-2)e)$, which is not index-realizable by \autoref{thm: (n, nl, n(n-2)l)}. On the other hand, if $a$ is even, then we have $a = 2e$ for some positive integer $e$. We know that $e \mid b$, so $b = ef$ for some positive integer $f$. We have $(a, b, c) = (2e, ef, 2e(e-1)f) = (e, e, e(e-1)) \cdot (2, f, 2f)$, where the two factors are index-realizable by \autoref{lem: (n, n, n(n-1))} and \autoref{lem: (a, b, ab)}, so $(a, b, c)$ is index-realizable.
		\item If $d = p$ is a prime and $p + 1 < a < 2p$, since $p \nmid a$, we see that our triple must have the form $(a, ae, ape)$. By \autoref{thm: (n, nl, npl)}, such a triple is index-realizable if and only if ${a \choose p} \mid ae$.
	\end{enumerate}
	Since we understand what happens for the above $d$, we may iterate through all values $1 \leqslant a \leqslant 10$ and look at all values $1 < d < a$ such that $d \nmid a$, which haven't been dealt with above. For $1 \leqslant a \leqslant 6$, there are no such values.
	
	\begin{itemize}
		\item For $a = 7$, we must look at $d = 4$. In this case, we have $7 \mid 4b$, so $b = 7e$ for some positive integer $e$. Our triple has the form $(7, 7e, 28e) = (7, 7, 28) \cdot (1, e, e)$. We know that $(7, 7, 28)$ is index-realizable by \autoref{prop: Geometric triples} using $(q, n, d) = (2, 3, 1)$. It follows that $(a, b, c)$ is index-realizable.
		\item For $a = 8$, we must look at $d = 3$. In this case, $(a, b, c) = (8, 8e, 24e)$. The triple $(8, 8, 24)$ is realized using $G = S_4$, with $H$ and $K$ being any two distinct Sylow $3$-subgroups, so by \autoref{prop: Products of Triples}, $(8, 8e, 24e)$ is index-realizable.
		\item For $a = 9$, we must look at $d = 4$ and $d = 6$. In these cases, we have triples of the form $(a, b, c) = (9, 9e, 36e)$ and $(a, b, c) = (9, 3e, 18e)$ respectively. In both cases, the triples are index-realizable by \autoref{prop: Products of Triples}:
		\begin{itemize}
			\item $(9, 9e, 36e) = (3, 3, 6)^2 \cdot (1, e, e)$
			\item $(9, 3e, 18e) = (3, 3, 6) \cdot (3, e, 3e)$
		\end{itemize}
		\item For $a = 10$, we must look at the values $d = 3, 4, 6$. The corresponding triples are $(a, b, c) = (10, 10e, 30e), (10, 5e, 20e), (10, 5e, 30e)$.
		\begin{itemize}
			\item The triple $(10, 10, 30)$ is index-realizable using $G = S_5$, $H = \Stab(\{1, 2\})$, and $K = \Stab(\{3, 4\})$. By \autoref{prop: Products of Triples}, so is $(10, 10e, 30e)$.
			\item The triple $(10, 5e, 20e)$ is index-realizable by \autoref{prop: Products of Triples} since it can be decomposed as $(5, 5, 10) \cdot (2, e, 2e)$.
			\item The triple $(10, 5, 30)$ is index-realizable by \autoref{thm: (n, nl, n(n-2)l)}. By \autoref{prop: Products of Triples}, so is $(10, 5e, 30e)$.
		\end{itemize}
	\end{itemize}
	This concludes the classification of index-realizable triples for $\min(a, b) \leqslant 10$.

	\newpage
    \bibliographystyle{alpha}
    \bibliography{biblio} 
	
\end{document}